\documentclass[12pt]{article}

\usepackage{graphicx}
\usepackage{enumerate}
\usepackage{natbib}

\usepackage{graphicx,bm,url,microtype}
\usepackage{paralist}
\usepackage{authblk}
\usepackage{amsfonts,amssymb,amsthm,amsmath, amsthm}

\usepackage{subcaption}
\usepackage{caption}
\usepackage[a4paper,text={16.5cm,25.2cm},centering]{geometry}
\usepackage[compact,small]{titlesec}

\usepackage{mathtools}

\newtheorem{thm}{Theorem}
\theoremstyle{remark}  
\newtheorem{remark}[thm]{Remark}
\newtheorem{prop}[thm]{Proposition}
\newtheorem{lem}[thm]{Lemma}

\usepackage{amsmath,amsthm,mathrsfs,amsfonts}
\usepackage{ulem} 	
\usepackage{url}
\usepackage{amsfonts,todonotes}
\usepackage{amssymb}
\usepackage{bbm}
\usepackage{graphicx}
\usepackage{subfig}
\usepackage{float}

\usepackage{multirow}

\usepackage{parskip}
\usepackage{booktabs}
\usepackage[shortlabels]{enumitem}
\usepackage{unicode}

\begin{document}

	\begin{center}
		\Large \bf Depth based trimmed means
	\end{center}
	\begin{center}
		Alejandro Cholaquidis $^\dagger$, \hspace{.2cm}
		Ricardo Fraiman $^\dagger$,  \hspace{.2cm}  Leonardo Moreno$^\ddagger$, Gonzalo Perera$^\S$\footnote{In memory: This work began with the collaboration of Gonzalo Perera who unfortunately passed away recently.} \\
		$^\dagger$ Centro de Matemáticas, Facultad de Ciencias, Universidad de la República, Uruguay. \\
 		$^\ddagger$ Instituto de Estadística, Departamento de Métodos Cuantitativos, FCEA, Universidad de la República, Uruguay.\\
        $^\S$ Centro Universitario Regional del Este - Universidad de la República, Uruguay.
	\end{center}

\begin{abstract}
	Robust estimation of location is a fundamental problem in statistics, particularly in scenarios where data contamination by outliers or model misspecification is a concern. In univariate settings, methods such as the sample median and trimmed means balance robustness and efficiency by mitigating the influence of extreme observations. This paper extends these robust techniques to the multivariate context through the use of data depth functions, which provide a natural means to order and rank multidimensional data. We review several depth measures and discuss their role in generalizing trimmed mean estimators beyond one dimension. 
	
Our main contributions are twofold: first, we prove the almost sure consistency of the multivariate trimmed mean estimator under mixing conditions; second, we establish a general limit distribution theorem  for a broad family of depth-based estimators, encompassing popular examples such as Tukey's and projection depth.
These theoretical advancements not only enhance the understanding of robust location estimation in high-dimensional settings but also offer practical guidelines for applications in areas such as machine learning, economic analysis, and financial risk assessment. 

A small example with simulated data is performed, varying the depth measure used and the percentage of trimmed data.
	
\end{abstract}

\section{Introduction}

Robust estimation of location is a cornerstone in statistics, especially when data are contaminated by outliers or when the underlying model is misspecified. In univariate analysis, the sample median is often preferred over the sample mean because of its superior breakdown point of \(1/2\), see \citet{Hampel1986}. However, despite its robustness, the median can be relatively inefficient under light-tailed distributions such as the normal distribution, see \citet{Huber1981}. To balance robustness and efficiency, trimmed means have been extensively studied and applied in practice, see \citet{Tukey1977} and  \citet{Wilcox1997}. 

Trimming, in essence, involves discarding a percentage of the extreme values from both ends of a dataset. This procedure is particularly useful when outliers might significantly distort results or lead to misleading conclusions. In practice, trimmed means are used across various fields. For instance, in economic analysis, this technique mitigates the undue influence of extreme events or anomalies that could skew indicators and trends, see \citet{wilcox2011}. 

It has been proven very recently that, for the one-dimensional case, the trimmed  achieves a Gaussian-type concentration around the mean, see \citet{oliveira2025}. Moreover, under slightly stronger assumptions, they also show that both tails of the trimmed mean satisfy a strong ratio-type approximation.  

In data science, trimming ensures that machine learning models are trained on cleaner, more representative data, thereby improving predictive performance, see \citet{rousseeuw1993}. Similarly, in financial risk assessment, trimming helps reduce the impact of anomalous data points, resulting in more reliable models for forecasting and decision-making, see \citet{barnett1994}. In medical research, trimming is applied to clinical datasets to eliminate outliers that may result from measurement errors or rare occurrences that do not reflect the broader population, see \citet{farcomeni2012}.

Extending these robust estimation ideas to the multivariate context presents additional challenges because there is no natural ordering in higher dimensions. In this setting, data depth functions have emerged as a powerful tool for ranking multivariate observations and constructing robust location estimators.

The concept of data depth was first introduced to statistics by Tukey in the 1970s, see \citet{tukey1975}, as a means of analyzing multivariate data. In \(\mathbb{R}^d\), the Tukey depth function, denoted by \(\textrm{TD}(x,P)\) for a probability measure \(P\), is defined as
\[
\textrm{TD}(x, P) \coloneqq \inf \{ P(H) : \text{\(H\) is a closed halfspace with } x \in H \}.
\]

In simple terms, a depth function measures how close a point is to the center of a distribution. The greater the depth, the more central the point. This concept is pivotal because it extends familiar ideas such as medians, ranks, and order statistics to multidimensional settings.

Building on Tukey's groundbreaking work, several alternative notions of depth have been developed. Two prominent examples are the simplicial depth and the spatial depth. The simplicial depth, introduced by  Liu, see  \citet{liu1990, liu1992, liu1992b}, is defined by
\[
\textrm{SD}(x,P) \coloneqq \mathbb{P}\big( x \in S[X_1, \ldots, X_{d+1}] \big),
\]
where \(X_1, \ldots, X_{d+1}\) are independent and identically distributed random vectors in \(\mathbb{R}^d\) with common distribution \(P\), and \(S[X_1, \ldots, X_{d+1}]\) denotes the \(d\)-dimensional simplex formed by these points, that is, the set of all convex combinations of the vertices. In contrast, the spatial depth, proposed by Serfling, see \cite{serfling2002,vardi2000}, is given by
\[
\textrm{SpD}(x,P)  \coloneqq 1 - \Big\|\mathbb{E}_{P}\big(S(x-X)\big)\Big\|,
\]
where for any nonzero vector \(x\), \(S(x)\) denotes its projection onto the unit sphere, with the convention that \(S(0)=0\), and $X$ is a random vector with distribution $P$.

In addition to these measures, other depth functions have been suggested for various types of data, including the  convex hull peeling depth, see  \citet{barnett1976}, Oja depth, see \citet{oja1983},  projection depth, see \citet{liu1992}, and spherical depth, see  \citet{elmore2006}. While some of these, like Tukey's depth, can be naturally extended to infinite-dimensional settings, others, such as Liu's simplicial depth, do not adapt as easily. In these cases, researchers have proposed alternative notions of depth (see, for example, \citet{fraiman2001, lopez2009, claeskens2014, cuevas2009}). More recently, the concept of depth has been explored within broader frameworks, including non-Euclidean metric spaces, see \citet{fraiman2019}. For a discussion on the essential properties that a depth function should satisfy (such as vanishing at infinity), see \citet{serfling2000}.

The concept of statistical depth has also been generalized to data on Riemannian manifolds and other metric spaces. For example, \citet{Cholaquidis2020} introduced the weighted lens depth---an extension of the classical lens depth that incorporates the geometry of Riemannian manifolds---and demonstrated its utility in supervised classification. In a subsequent work, \citet{CholaquidisFraimanMoreno2021} examined level sets of depth measures  providing uniform consistency results and new insights into central dispersion in general metric spaces.

An early approach to the trimmed mean for multivariate data was introduced by \citet{gordaliza1991}, who proposed impartial trimming---a data-driven procedure for discarding a small fraction of observations deemed ``outlying.'' This approach laid the foundation for robust clustering methods, such as the impartial trimmed \(k\)-means method introduced by \citet{CuestaAlbertos1997}. The robustness of this method was later examined in detail by \citet{GarciaEscudero1999a}, demonstrating its effectiveness in the presence of outliers. Subsequently, \citet{escudero2008} extended the clustering procedure to adjust clusters with varying dispersion and weight. More recently, the concept of impartial trimming has been adapted to functional data analysis, see   \citet{CuestaAlbertosFraiman2007}, thereby broadening its applicability.

The extension of L-estimators to the multivariate case, as discussed in \citet{fraiman1999}, provided a powerful tool for generalizing the concept of trimmed means to higher-dimensional contexts---even for infinite-dimensional data, see also \citet{fraiman2001}. The key idea is that depth measures in multivariate data facilitate the generalization of the trimmed mean concept to \(\mathbb{R}^d\), as proposed in \citet{Zuo2002}. In that work it is obtained 
the consistency and asymptotic normality of a smoothed version of the trimmed mean estimator, albeit under restrictive assumptions regarding the depth measure and the underlying data distribution,   see Remark \ref{zuo} for a detailed comparison with our results. 

Subsequent research has extended these ideas, investigating various properties of multivariate trimmed means---such as robustness, consistency, and limiting distributions---for specific depth measures. For example, \citet{Zuo2006} focuses on projection depth, showing that it leads to a Gaussian limit distribution, while \citet{masse2004, masse2009} studied Tukey’s depth. In   \citet{stigler1973asymptotic} it is proven that in the one-dimensional case, the limit distribution can be non-Gaussian.  In particular, in \cite{ilienko2025}, the authors explore empirical Tukey depth and prove strong limit theorems similar to the Marcinkiewicz-Zygmund law. They study the convergence of empirical depth-trimmed regions to the theoretical ones, focusing on uniform distributions on convex bodies, with convergence rates of \( (n^{-1} \log \log n)^{1/2} \).
These results highlight the diverse behaviours of multivariate trimmed means, depending on the depth measure employed.

The main contribution of our work is twofold. First, we prove the almost sure consistency of the multivariate trimmed mean estimator in the case of mixing data. Second, we obtain the asymptotic distribution of the trimmed mean, for a broad family of depth measures, which includes Tukey's depth and projection depth as special cases,   see also Remark \ref{zuo} for a detailed comparison of this result and a quite restrictive one, existing in the literature.

%
%

\subsection{The multivariate trimmed mean}

In what follows, we investigate multivariate trimmed means defined via general depth functions under relatively weak conditions. Specifically, we assume the existence of a depth function,
$	D: \mathbb{R}^d \to (0, T],$ and consider a sequence $\{\hat{f}_n\}_n$ of density estimators, that are assumed to be continuous functions. For instance, given a suitably chosen continuous kernel $K_{h_n}$ which depends on a bandwidth $h_n$,  $\hat{f}_n(x)=(K_{h_n} * F_n)(x)$, where $F_n$ is the empirical distribution function and $*$ denotes the convolution.

\noindent
Let $\mathcal{F}$ the set of continuous and bounded functions on $\mathbb{R}^d$.  Consider the map,
\[
D=D(x,g) :\;\mathbb{R}^d \times \mathcal{F}\;\longrightarrow\;\mathbb{R},
\]
where $D(x,g)$ is the depth corresponding to the density $g$.  The multivariate trimmed mean is defined as
\begin{equation}\label{popversion}
	\Pi(D,f):= \int_{\mathbb{R}^d } t \ \mathbb I_{\{D(t)\geq a\}} f(t)dt,
\end{equation}
while, if $D_n$ denotes an estimator of $D$ based on a sample $X_1,\dots,X_n$ of a random vector with density $f$,  the empirical counterpart of $\Pi$ is 
\begin{equation}\label{estimador}
	\Pi(D_n, \hat f_n):= \int_{\mathbb{R}^d } t \ \mathbb I_{\{D_n(t)\geq  a\}}\hat f_n(t)dt.
\end{equation}

To prove the almost sure (a.s.) consistency of \eqref{popversion} $D_n$ can be any estimator that converges a.s. to $D$, see \textbf{(H'4)} below. However, the limiting distribution is obtained for a particular choice of $D_n$. We replace the empirical version of the depth function $
\tilde{D}_n(x)=\,D(x,F_n)$ by a smooth version given by $D_n(x)=\,D(x,\hat{f}_n)
$. 

For  instance, for  Liu's simplicial depth we can take a regular version for the empirical version, i.e.
\[
D(x,g) \;=\; P_G\bigl(x \in \text{Simplex}\bigr),
\quad
\text{where $G$ has density $g$.}
\]
That is
$
D(x,g) = P_G\bigl(x \in \text{Simplex}\bigr),
$
as defined  in Liu's paper, but,
$
D_n \;\text{will be}\;P_{G_n}\bigl(x \in \text{Simplex}\bigr)$ with
$G_n \sim \hat{f}_n,
\;\text{instead of}\;
P_{F_n}\bigl(x \in \text{Simplex}\bigr)
$.

The main questions we will address correspond to strong consistency and the asymptotic distribution of $\Pi(D_n, \hat f_n)$. 
 
\begin{remark} \label{zuo}
	Although a general Central Limit Theorem can be found in \citet{Zuo2002} (see Theorem 3.1), the assumptions taken in that work include the Hadamard differentiability of the operator \eqref{popversion} (see hypothesis C4 in the mentioned document). In our work, this is not assumed as a hypothesis but is instead obtained the asymptotic distribution \eqref{estimador} as a consequence of the regularity properties of \(D\) (see hypotheses \textbf{(H2)} and \textbf{(H3)} below), which constitutes a considerable leap in generality. Furthermore, hypothesis C5 of \citet{Zuo2002} requires that the operator \eqref{popversion} be close to \eqref{estimador} at a rate proportional to \(F_n-F\). Again, here that hypothesis is not required; instead, such consistency is derived separately. 	Notably, until now, this work has been the only study in the literature addressing the asymptotic distribution of the estimator in the general case. Furthermore, in the aforementioned paper the asymptotic distribution is always normal— a property that does not generally hold, see \citet{stigler1973asymptotic}. For Tukey depth, a  an asymptotic distribution  Theorem  is established in \citet{masse2004} when the trimmed mean (see Equation \ref{popversion}) is replaced by a smoothed version, in which the indicator function is substituted with a continuously differentiable weight $W$. A similar result was previously obtained for simplicial depth in \citet{dumbgen1992limit}, again by replacing the trimmed mean with a smoothed version.
\end{remark}

\subsection{Road map}

The remainder of the paper is organized as follows.  In Section \ref{geometri} we study some topological and geometric properties of the level sets of $D$ which will be used through the manuscript. Specially to get the limiting distribution Theorem for the trimmed means.  The almost sure consistency of the estimator is obtained in Section \ref{consistency}.  Section \ref{hadd} is quite technical and it is devoted to prove the differentiability of the functional $\Pi(D,f)$ w.r.t the second component, in the Hadamard sense. This will be a key ingredient to get the limiting distribution Theorem, obtained in Section \ref{tclsec}. At the end of this Section, a small simulation example is conducted to illustrate the shape of the limiting distribution.

\section{Main hypotheses}

Let  $a\in (0,T)$ fixed, and a depth $D: \mathbb{R}^d \to (0,T]$.
The following set of hypotheses will be considered through the  manuscript.

\begin{itemize}
	\item[\textbf{(H1)}] $D$ is continuous in $\mathbb{R}^d$,  and $	\lim_{\|t\| \to +\infty} D(t) = 0.$
	\item[\textbf{(H2)}] There exists $\delta>0$ and a finite number of points $\mu_1,\dots,\mu_\ell$, with $\mu_i \in \mathbb{R}^d$ for all $i=1,\dots,\ell$, such that:
	 $D(\mu_i)> a$. Also, for each $\mu_i$ and each $\vec{n}\in S^{d-1}$ there exists a real value $\gamma_i(a,\vec{n})\in \mathbb{R}^+$,  such that,	the functions 
	$$\lambda\in (-\delta,\delta) \mapsto D(\mu_i+(\lambda+\gamma_i(a,\vec{n}))\vec{n}),$$
	are strictly decreasing, and  $\gamma_i(a,\vec{n})$ is such that $D(\mu_i+\vec{n}\gamma_i(a,\vec{n}))=a$.  As we will see \textbf{(H2)} allows to define  $\ell$ functions $\tau_i: \Omega_i\subset (0,T]\times S^{d-1}\to \mathbb{R}^d\setminus\{\mu_i\},$
	by setting
	\[
	\tau_i(s,\vec{n}) = \mu_i + \bigl(\lambda_s+\gamma_i(a,\vec{n})\bigr)\vec{n},
	\]
	where $\lambda_s\in(-\delta,\delta)$ is the unique value for which $D (\tau_i(s,\vec{n}))=s$. See Section \ref{geometri} for details.
	\item[\textbf{(H3)}] $D$ is of class $\mathcal{C}^p$, $1 < p \leq \infty$, and for all $i=1,\dots,\ell$ $d\tau_i^{-1}$ has no singular points in $\tau_i(\Omega_i)$ (i.e., $d_x\tau_i^{-1}$ has full rank $\forall x\in  \tau_i(\Omega_i)$, $\Omega_i$ as in \textbf{(H2)}).
	\item[\textbf{(H4)}] \( f \)  is probability density, continuos on an open set $\mathcal{U}$, containing $\{t:D(t)\geq a\}\cup_i \tau_i(\Omega_i)$.
\end{itemize}

\begin{remark}\label{remhipo}

 Hypothesis \textbf{(H1)} holds for a wide range of depth measures, including lens depth, simplicial depth, L1-depth, and Tukey depth, among many others. In these cases, no additional assumptions on $X$ are required beyond the existence of a density, see \citet{serfling2000}. Hypothesis \textbf{(H2)} is satisfied, for example, when $X$ is symmetric with respect to a single point (see \citet{serfling2000}, p. 464) for simplicial depth, Tukey depth, and L1-depth. 	Moreover, this hypothesis is considerably more general since it permits a finite number of modes. 	Hypothesis \textbf{(H3)} is used to  get the asymptotic distribution in  Theorem \ref{tcl} below. To this aim \textbf{(H3)} for  $p=2$ is enough. This ensures that each $\tau_i$ is a diffeomorphism, see Lemma \ref{lemdifeo}. 	Hypothesis \textbf{(H4)} is  a quite weak assumption that includes most well-known densities.

\end{remark}

\section{On the level sets of $D$}  	\label{geometri}

Given $a\in (0,T)$, we associate to $a$ its nearest centre $\mu_i$, as introduced in \textbf{(H2)}); if two or more centres are equidistant from $a$, one is chosen arbitrarily. We then define
\[
\tau_i(a,\vec{n}) = \mu_i + \gamma_i(a,\vec{n})\,\vec{n},
\]
so that, by definition, $D\bigl(\tau_i(a,\vec{n})\bigr)=a$. Note that, under hypothesis \textbf{(H2)}, the mapping
\[
\lambda\in(-\delta,\delta) \longmapsto D\Bigl(\mu_i+(\lambda+\gamma_i(a,\vec{n}))\vec{n}\Bigr),
\]
is strictly decreasing. Consequently, if $\lambda\neq 0$, then
$D(\mu_i+(\lambda+\gamma_i(a,\vec{n}))\vec{n})\neq a.$

More generally, we define the functions $\tau_i: \Omega_i\subset (0,T]\times S^{d-1}\to \mathbb{R}^d\setminus\{\mu_i\},$
by setting
\[
\tau_i(s,\vec{n}) = \mu_i + \bigl(\lambda_s+\gamma_i(a,\vec{n})\bigr)\vec{n},
\]
where $\lambda_s\in(-\delta,\delta)$ is the unique value for which $D (\tau_i(s,\vec{n}))=s$.

Since $D$ is a continuous function and $D(\mu_i)>a$ for all $i=1.\dots, \ell$, we can take $\delta$ small enough such that 
\begin{equation}\label{menormu}
	D(\mu_i+(-\delta+\gamma_i(a,\vec{n}))\vec{n}))<\mu_i
\end{equation}
for all $\vec{n}$, and 
\begin{equation}\label{disjoint}
	\tau_i(\Omega_i)\cap \tau_j(\Omega_j)=\emptyset\qquad \forall i\neq j.
\end{equation} 

In what follows we assume that $\delta$ is small enough to guarantee \eqref{menormu} and \eqref{disjoint}.

\begin{lem}  For each $i=1,\dots,\ell$, $\tau_i$ is a homeomorphism over its image, $\tau_i(\Omega_i)$,  whose inverse is given by:
	\begin{equation}\label{eq1}
		\tau_i^{-1} (x) = \left( D(x), \frac{x - \mu_i}{\| x - \mu_i \|} \right), \quad \forall x \in \tau_i(\Omega_i).
	\end{equation} 
\end{lem}	
\begin{proof} We remove, for ease of writing, the dependence on $i$ on the subindexes.
	
	Let $\nu$ be the function defined by \eqref{eq1} i.e.,
	\[
	\nu (x) = \left( D(x), \frac{x - \mu}{\| x - \mu \|} \right).
	\]
	Observe that, by \eqref{menormu}, $\mu\notin \tau(\Omega)$.
	If $x \neq \mu$, then $\tau(\nu(x)) = x$, then 
	$D(y)=D(x)$, and there exists $\lambda$, such that, 
	$$y=\mu+\lambda\frac{x-\mu}{\|x-\mu\|}.$$
	But, from 
	$$x=\mu+\|x-\mu\|\frac{x-\mu}{\|x-\mu\|}$$
	and \textbf{(H1)}, it follows that $\lambda=\|x-\mu\|$. Then $x=y$.
	This proves that $\tau(\nu(x))=x$ and, 
	\begin{equation}\label{A}
		\tau \circ \nu = \text{Id} (\tau(\Omega)) ,
	\end{equation}
	
	where $\text{Id}(\cdot)$ denotes the identity map.
	
	On the other hand, if $\vec{n} \in S^{d-1}$, then by definition of $\nu$:
	\[
	\nu (\tau(a, \vec{n})) = \left(a, \frac{\tau(a, \vec{n}) - \mu}{\| \tau(a, \vec{n}) - \mu \|} \right) = (a, \vec{n}),
	\]
	which implies:
	\begin{equation}\label{B}
		\nu \circ \tau = \text{Id} (\Omega).
	\end{equation}
	From \eqref{A} and \eqref{B} we deduce that $\tau:\Omega\to \tau(\Omega)$ is bijective and that $\tau^{-1} = \nu$.
	
	Since $\tau^{-1}$ is bijective and continuous, and its domain is an open set of $\mathbb{R}^d$, by the Invariance of the Domain Theorem, see \citet{engelking1992topology}, we conclude that $\tau^{-1}$ is a homeomorphism.
\end{proof}

\begin{lem}\label{lem3} If $C_a = \{ x \in \mathbb{R}^d:D(x) = a \}$ , then $C_a$ is a union of $\ell$ disjoint sets, $C_a^1,\dots,C_a^\ell$, each of them homeomorphic to $S^{d-1}$.  
\end{lem}	
\begin{proof} 
	It is clear that
	\[ C_a = \cup_{i=1}^\ell \{ x \in \mathbb{R}^d : x = \tau_i(a,\vec{n}) \text{ for some } \vec{n} \in S^{d-1} \}=\cup_{i=1}^\ell C_a^{i}. \]
	Since $0 < a < D(\mu_i)$  $C_a^i \neq \emptyset$  and, from \eqref{menormu} $\mu_i \notin C_a^i$ . \\
	Consider $\varphi_a^i: S^{d-1} \to C_a^i$ given by $\varphi_a^i(\vec{n}) = \tau_i(a,\vec{n})$, which implies $\varphi_a$ is a continuous bijection, with
	\[ (\varphi_a^i)^{-1}(x) = \frac{x - \mu_i}{\| x - \mu_i \|}, \]
	continuous, hence $\varphi_a^i$ is a homeomorphism.  By \eqref{disjoint} the sets $C_a^i$ are disjoint.
\end{proof}

\begin{lem} \label{lemdifeo}
	Under \textbf{(H1)} to \textbf{(H3)}, $\tau_i$ is a $\mathcal{C}^p$-diffeomorphism  and $C_a^i$ is $\mathcal{C}^p$-diffeomorphic to $S^{d-1}$ for all $i=1,\dots,\ell$.
\end{lem}	
\begin{proof} 
	
	Since $\tau_i^{-1}$ is of class $\mathcal{C}^p$ in $\tau_i(\Omega_i)$ open, $\tau_i^{-1}$ is bijective and $d\tau_i^{-1}$ has no singular points. By the inverse mapping Theorem, see \citet{lang2012real}, $\tau_i^{-1}$ is a diffeomorphism of class $\mathcal{C}^p$, hence so is $\tau_i$.  For $C_a^i$,   both $\varphi_a^i(\cdot) = \tau_i(a,\cdot)$ and
	\[ (\varphi_a^i)^{-1}(\cdot) = \frac{\cdot - \mu}{\| \cdot - \mu \|}, \]
	are $\mathcal{C}^p$.  
	
\end{proof}

\section{Consistency}\label{consistency}

In this section we consider, as before, a data–depth function $D:\mathbb{R}^d \to (0,T]$,  $a\in (0,T)$, and  any estimator of \(D\), denoted by \(D_n\), which is constructed from a sample \(X_1, X_2, \dots, X_n\) of a random vector \(X\). We assume that the common distribution of \(X\) has a density \(f\) but do not require the sample to be iid.   In fact, the consistency results will be obtained from hypothesis \textbf{(H2)} and the following four weak hypotheses.  

\begin{itemize}
	\item[\textbf{(H'1)}] \(\mathbb{E}(\|X\|) < \infty\).
	
	\item[\textbf{(H'2)}] Let \(\hat{f}_n\) be a sequence of estimators of \(f\) (based on \(X_1, X_2, \dots, X_n\)) such that
	\[
	\hat{f}_n \to f \quad \text{almost surely, uniformly on compact sets.}
	\]

	\item[\textbf{(H'3)}] \(\mathbb{P}\bigl(D(X) = a\bigr) = 0\).
	
	\item[\textbf{(H'4)}] The estimator \(D_n\) is such that
	\[
	S_n := \sup_{x\in \mathbb{R}^d} \bigl|D_n(x) - D(x)\bigr| \xrightarrow{\text{a.s.}} 0.
	\]
	where, in what follows,  $\xrightarrow{\text{a.s.}}$ will denote almost sure convergence.

\end{itemize}

\begin{remark}
	\begin{enumerate}
		\item By considering a kernel–based density estimator, hypothesis \textbf{(H'2)} is satisfied, see \citet{silverman1978}. Even Property \textbf{(H'2)} is satisfied under some conditions of sample dependence, see \citet{roussas1991, gine2002, einmahl2005}. 
		\item Hypothesis \textbf{(H'3)}  is met if \(D(X)\) has a density or, as follows from Lemma~\ref{lem3}, if \(X\) has a density and hypothesis \textbf{(H1)} and \textbf{(H2)} holds.  
		\item if \(D_n\) is the empirical version of \(D\) (based on an iid sample), then uniform convergence \textbf{(H'4)} is ensured for many depth notions, see Remark A.3 of \citet{serfling2000}.  \citet{donoho1992} proved it for the sample half-space depth (Tukey's depth). Similarly, they prove the convergence for the simplicial depth, see \citet{liu1990}, \citet{dumbgen1992limit}, and \citet{arcones1993}. Moreover, other works show more depth measures where the hypothesis \textbf{(H'4)} is satisfied (for the iid case), see \citet{CholaquidisFraimanMoreno2021, liu1993,zuo2000b}.   For the simplicial depth \textbf{(H'4)} holds even when $X_1,\dots,X_n$ are not iid, from Theorem 1 in \citet{dumbgen1992limit}, when $D_n(x)=D(x,P_n)$ being $P_n$ any sequence such that $P_n\to P$ weakly. The same holds for Tukey's depth, see \citet{donoho1992},   pp. 1816--1817)
	\end{enumerate}
	
\end{remark}

The following theorem  states that the trimmed mean estimator \eqref{estimador} converges almost surely to \eqref{popversion}.

\begin{thm} Assume \textbf{(H2)} and \textbf{(H'1)} to \textbf{(H'4)}. Then, 
	$$\Pi(D_n,\hat{f}_n)\xrightarrow{\text{a.s.}} \Pi(D,f)\qquad \text{as } n\to \infty.$$ 
\end{thm} 
\begin{proof}
	Let   \[
	\Pi(D,\hat{f}_n)
	\; =\; \int_{\mathbb{R}^d} t\,\mathbf{1}_{\{D(t)\geq a\}}\, \hat{f}_n(t)dt.
	\]
	Then, 
	\[
	\Pi(D,\hat{f}_n) - \Pi(D,f)
	\;=\; \int_{\mathbb{R}^d} t\,\mathbf{1}_{\{D(t)\geq a\}}\, \hat{f}_n(t)dt
	\;-\; \int_{\mathbb{R}^d} t\,\mathbf{1}_{\{D(t)\geq a\}}\,f(t)dt.
	\]
	
	\bigskip
	
	\noindent
	From  \textbf{(H'2)} and \textbf{(H2)} $	 \Pi(D,\hat{f}_n) \;\xrightarrow{\text{a.s.}}\; \Pi(D,f)$ as $n \to \infty$. 
	Let us write, 
	
	$$\|\Pi(D_n,\hat{f}_n)- \Pi(D,f)\|\leq \|\Pi(D_n,\hat{f}_n)- \Pi(D,\hat{f}_n)\|+\| \Pi(D,\hat{f}_n)- \Pi(D,f)\|.$$
	Let us bound, 
	$$\left\|\Pi(D_n,\hat{f}_n) -  	 \Pi(D,\hat{f}_n)\right\|\leq  	\int_{\mathbb{R}^d} \|t\||\mathbf{1}_{\{D(X_i)\geq a\}}-\mathbf{1}_{\{D_n(X_i)\geq a\}}|\,\hat{f}_n(t)dt. $$

	From \textbf{(H'4)},  $D(t) - S_n \;\le\; D_n(t) \;\le\; D(t) + S_n.$  	Hence,
	\begin{multline*}
		\|\Pi(D_n,\hat{f}_n)  -  	 \Pi(D,\hat{f}_n)\|  \leq
		\int_{\mathbb{R}^d}\|t\||\mathbf{1}_{\{(D(t) + S_n)\geq a\}}- \mathbf{1}_{\{(D(t) - S_n)\geq a\}}|\mathbf{1}_{\{S_n\leq \delta\}}\hat{f}_n(t)dt \ + \\
		\int_{\mathbb{R}^d} \|t\||\mathbf{1}_{\{(D(t) + S_n)\geq a\}} - \mathbf{1}_{\{(D(t) - S_n)\geq a\}}|\mathbf{1}_{\{S_n\geq \delta\}}\hat{f}_n(t)dt=A(\delta)+B(\delta)
	\end{multline*}
	On one hand,   
	$$B(\delta)\leq  \int_{\mathbb{R}^d} \|t\|\hat{f}_n(t)dt\mathbf{1}_{\{S_n\geq \delta\}}\xrightarrow{\text{a.s.}} 0,$$
	since $\mathbf{1}_{\{S_n\geq \delta\}}\xrightarrow{\text{a.s.}} 0$, and 
	$$\int_{\mathbb{R}^d} \|t\|\hat{f}_n(t)dt \xrightarrow{\text{a.s.}} \mathbb{E}\|X_1\|.$$

	On the other hand,
	\begin{align*}
		A_n(\delta)\leq & \int_{\mathbb{R}^d}\|t\||\mathbf{1}_{\{(D(t)+\delta)\geq a\}}-\mathbf{1}_{\{(D(t)-\delta)\geq a\}}|\hat{f}_n(t)dt \\
		\to &\mathbb{E}\Big(\|X_1\||\mathbf{1}_{\{(D(X_1)+\delta)\geq a\}}-\mathbf{1}_{\{(D(X_1)-\delta)\geq a\}}|\Big):=A(\delta), \quad \text{ as }n\to \infty.
	\end{align*}
	Then,  $\overline{\lim} A_n(\delta)=A(\delta)$. 	Let us prove that $A(\delta)\to 0$ as $\delta\to 0^+$. We will use Dominated Convergence Theorem. 	First, $\|X_1\||\mathbf{1}_{\{(D(X_1)+\delta)\geq a\}}-\mathbf{1}_{\{(D(X_1)-\delta)\geq a\}}|\leq \|X_1\|\in L^1.$ 
	From \textbf{(H'3)}, with probability one, 
	$$\mathbf{1}_{\{(D(X_1)+\delta)\geq a\}}-\mathbf{1}_{\{(D(X_1)-\delta)|\geq a\}}= 0\qquad \text{ as } \delta \to 0^+.$$
	Lastly, $A(\delta)\to 0$ as $\delta \to 0^+$.

\end{proof}

\section{On the Haddamard differentiability of trimmed means.}\label{hadd}

Let us consider 
\[
\mathcal{D}
\;=\;
\{\,y:\mathbb{R}^d \to \mathbb{R}\;\text{continuous and bounded}\},
\]
and $\mathcal{U}$ as in \textbf{(H4)}. Define,
\[
\mathcal{D}\mathcal{G}
\;=\;
\{\,g:\mathbb{R}^d \to \mathbb{R}\;\text{continuous on } \mathcal{U}, \text{ bounded },\;g\in L^1(\mathbb{R}^d)\}.
\]
Let $D\in\mathcal{D}$ and $f\in\mathcal{D}\mathcal{G}$ such that \textbf{(H1)} to \textbf{(H4)} hold.
Let $0<a<T$ and $a_1 < a < a_2 < T$  such that   $K=\{\,t\in\mathbb{R}^d : a_1 \le D(t) < a_2\}\subset \cup_{i=1}^\ell \Omega_i$ and \eqref{menormu} and \eqref{disjoint} holds. Then,
\[
\mathcal{T}(\tilde{y},g)
\;=\;
\int_{K}
t\,
1_{\{\,\tilde{y}(t)>a\}}
\;g(t)\;dt
\quad
\text{for any }(\tilde{y},g)\in \mathcal{D}\times \mathcal{D}\mathcal{G}.
\]
The aim of this section is prove that $\mathcal{T}$ is  Hadamard differentiable at $(D,f)$ and 
\[
\mathcal{T}'(\tilde{y},g)
\;=\;
\sum_{i=}^\ell\Bigg\{ \int_{S^{d-1}}
\tau_i(a,\vec{n})\,
f\bigl(\tau_i(a,\vec{n})\bigr)\,
\left|\det\!\bigl(d_{(a,\vec{n})}\tau_i\bigr)\right|\,
\tilde{y}\bigl(\tau_i(a,\vec{n})\bigr)
\,d\vec{n}
\;+\;
\int_{\mathbb{R}^d}
t\,
\mathbf{1}_{\{D(t)\geq a\}}
\,g(t)\,dt\Bigg\},
\]
see Theorem \ref{thhaddamard} below.  

We will prove the result for $\ell=1$ but the general proof follows exactly the same ideas, so in what follows we denote just $\tau$ instead of $\tau_i$. The proof is based on some technical results. 

\begin{lem} \label{lem5}
	Let $G: I \to \mathbb{R}^d$ be a continuous and bounded function on the bounded interval $I \subset \mathbb{R}$; assume that $H: I \to \mathbb{R}$ is uniformly continuous and set, for $a \in \overset{\circ}{I}$:
	\[
	V_{\varepsilon}(a) = \int_I G(\delta) \left( \mathbf{1}_{\{ H(\delta) \geq \frac{a-\delta}{\varepsilon} \}} - \mathbf{1}_{\{ \delta \geq a \}} \right) \frac{d\delta}{\varepsilon}.
	\]
	Then, $\lim_{\varepsilon \to 0^+} V_{\varepsilon}(a) = G(a) H(a).$
\end{lem}	

\begin{proof} Pick $\eta > 0$; take $I_1, \dots, I_{N_{\eta}}$ as a partition of $I$ into disjoint intervals such that
	\[
	\sup_{x, y \in I_i} |H(x) - H(y)| < \eta, \quad \forall 1 \leq i \leq N_{\eta},
	\]
	and that $a \in \overset{\circ}{I_i}$ is an interior point of $I_i$. Define $H_i$ as the value of $H$ at an arbitrary point of $I_i$ and set:
	\[
	H^\eta = \sum_{i=1}^{N_{\eta}} H_i \mathbf{1}_{I_i},
	\qquad \text{ and }\qquad 
	V_{\varepsilon}^{\eta}(a) = \int_I G(\delta) \left( \mathbf{1}_{\{ H^{\eta}(\delta) \geq \frac{a-\delta}{\varepsilon} \}} - \mathbf{1}_{\{ \delta \geq a \}} \right) \frac{d\delta}{\varepsilon}.
	\]
	Then,
	\[
	V_{\varepsilon}^{\eta}(a) = \sum_{i=1}^{N_{\eta}} \int_{I_i} G(\delta) \left( \mathbf{1}_{\{ H_i \geq \frac{a-\delta}{\varepsilon} \}} - \mathbf{1}_{\{ \delta \geq a \}} \right) \frac{d\delta}{\varepsilon}.
	\]
	For $0 < \varepsilon < \varepsilon_{\eta}$ such that $(a - \varepsilon_\eta(\|H\|_\infty + \eta) ,a+\varepsilon_\eta(\|H\|_\infty + \eta))\subset I_{i_a}$,
	\[
	V_{\varepsilon}^{\eta}(a) = \sum_{i=1}^{N_{\eta}} \int_{I_i \cap [a - \varepsilon H_i, a]} G(\delta) \frac{d\delta}{\varepsilon}.
	\]
	
	Observe that, 
	$$\int_{I_{i_a}\cap [a-\varepsilon H_{i_a},a)} \frac{G(\delta)}{\epsilon}d\delta=\int_{a-\varepsilon H_{i_a}}^a\frac{G(\delta)}{\epsilon}d\delta=-H_{i_a}\int_{a-\varepsilon H_{i_a}}^a \frac{G(\delta)}{-\epsilon H_{i_a}}d\delta\to H_{i_a}G(a)=H^\eta(a)G(a). $$ 
	Then, 
	\begin{equation}
		V_{\varepsilon}^{\eta}(a) \longrightarrow H^\eta(a) G(a) \quad \text{as } \varepsilon \to 0^+. \tag{A}
	\end{equation}
	It is obvious that:
	\begin{equation}
		H^\eta(a) G(a) \longrightarrow H(a) G(a) \quad \text{as } \eta \to 0^+. \tag{B}
	\end{equation}
	In addition: 
	\begin{align*}
		\|V_{\varepsilon}(a) - V_{\varepsilon}^{n}(a) \| &=\Bigg\| \int_I G(\delta) \left( \mathbf{1}_{\{H(\delta) > \frac{a - \delta}{\varepsilon}\}} - \mathbf{1}_{\{H^\eta(\delta) > \frac{a - \delta}{\varepsilon}\}} \right) d\delta\Bigg\|  \\
		&\leq \int_I \|G(\delta)\| \frac{\mathbf{1}_{\left\{ \delta > a - \varepsilon H(\delta) \right\} \Delta \left\{\delta > a - \varepsilon H^\eta(\delta) \right\}}}{\varepsilon} d\delta \\
		&= \int_I \|G(\delta)\| \frac{\mathbf{1}_{\left\{ a - \varepsilon H(\delta) > \delta > a - \varepsilon H^\eta(\delta) \right\}}}{\varepsilon} d\delta \\
		&\hspace{2cm}+ \int_I \|G(\delta)\| \frac{\mathbf{1}_{\left\{ a - \varepsilon H^\eta(\delta) > \delta > a - \varepsilon H(\delta) \right\}}}{\varepsilon} d\delta \\
		&\leq \int_I \|G(\delta)\| \frac{\mathbf{1}_{\left\{ a - \varepsilon (H^\eta(\delta) - \eta) > \delta > a - \varepsilon H^\eta(\delta) \right\}}+\mathbf{1}_{\left\{ a - \varepsilon H^\eta(\delta) > \delta > a - \varepsilon (H^\eta(\delta) + \eta) \right\}}}{\varepsilon} d\delta \\
		&= \sum_{i=1}^{N_{\eta}} \int_{I_i} \|G(\delta)\| \frac{\mathbf{1}_{\left\{ a - \varepsilon H_i + \varepsilon \eta > \delta > a - \varepsilon H_i - \varepsilon \eta \right\}}}{\varepsilon} d\delta \\
		&= \int_{I_i^a\cap  (a - \varepsilon H_{i_a} - \varepsilon \eta, a - \varepsilon H_{i_a} + \varepsilon \eta)} \frac{\|G(\delta)\|}{\varepsilon} d\delta \hspace{1cm}  \text{for} \ \ 0 < \varepsilon < \varepsilon_{\gamma}^*\\	
		&= (\eta - H_{i_a}) \int_{a}^{a+(\eta - H_{i_a}) \varepsilon} \|G(\delta)\| \, d\delta 
		+ (H_{i_a} + \eta) \int_{a}^{a - \varepsilon (\eta + H_{i_a})} \frac{\|G(\delta)\|}{-\varepsilon (\eta + H_{i_a})} \, d\delta \\ &\longrightarrow_{\varepsilon \to 0^+} 	\|G(\delta)\| \bigl((\eta - H_{i_a}) + (H_{i_a} + \eta)\bigr) = 2\eta \|G(a)\|.
	\end{align*}
	Then, 
	\begin{equation}
		\lim_{\varepsilon \to 0^+} \|V_{\varepsilon}(a) - V_{\varepsilon}^{n}(a) \| \leq 2\eta \|G(a)\|. \tag{C}
	\end{equation}
	From (A), (B), and (C), we get:
	\begin{align*}
		\overline{\lim_{\varepsilon \to 0^+}} \|V_{\varepsilon}(a) - H(a) G(a) \| 
		&\leq \overline{\lim_{\varepsilon \to 0^+}}  \|V_{\varepsilon}(a) - V_{\varepsilon}^{\eta}(a) \| \\
		&\qquad + \overline{\lim_{\varepsilon \to 0^+}}  \|V_{\varepsilon}^{\eta}(a) - H^\eta(a) G(a) \| 
		+\overline{\lim_{\varepsilon \to 0^+}}  \|H^\eta(a) G(a) - H(a) G(a)\| \\
		&\leq 2\eta \|G(a)\| + 0 + 0.
	\end{align*}
	
	Taking \(\eta \downarrow 0^+\), we get: $\lim_{\varepsilon \to 0^+} \|V_{\varepsilon}(a) - H(a) G(a) \| = 0.$ Thus, the Lemma follows.
	
\end{proof}

\begin{prop} \label{prop1} Assume that  \textbf{(H1)} to \textbf{(H4)}  hold.  Let \(0 < a < T\), \(r > 0\), and $\mathcal{U}$ as in \textbf{(H4)}. Define,
	\[
	\mathcal{B}(r) = \{ Y: \mathbb{R}^d \to \mathbb{R} \text{ continuous on }\mathcal{U} \text{ and bounded, with } \|Y\|_{\infty} \leq r \},
	\]
	and
	\[
	\Delta_{\varepsilon}(Y) = \int_{\mathbb{R}^d} t \left( \frac{\mathbf{1}_{\{|D(t) + \varepsilon Y(t)| \geq a\}} - \mathbf{1}_{\{|D(t)| \geq a\}}}{\varepsilon} \right) f(t) dt.
	\]
	Then,
	\[
	\lim_{\varepsilon \to 0^+}  \sup_{Y \in \mathcal{X}} \| \Delta_{\varepsilon}(Y) - \Delta(Y) \| =0, \qquad \text{where } \mathcal{X} \text{ is any } \|\cdot\|_{\infty} \text{-compact subset of } \mathcal{B}(r),
	\]
	and
	\begin{equation*}
		\Delta(Y) = \int_{S^{d-1}} \tau(a, \vec{n}) \left| \det(d\tau(a, \vec{n})) \right| Y(\tau(a, \vec{n})) f(\tau(a, \vec{n})) d\vec{n}.
	\end{equation*}
\end{prop}
\begin{proof}
If $Y \in \mathcal{B}(\tau)$ and $D(t) < a - \varepsilon \tau$, then $D(t) + \varepsilon Y(t) < a$,
	and if $D(t) > a + \varepsilon \tau$, then $D(t) + \varepsilon Y(t) > a$. In both cases,
	the integrator in $\Delta_\varepsilon(Y)$ equals 0. By taking $\varepsilon$ small enough,
	we can reduce the integral in $\Delta_\varepsilon(Y)$ to the compact set:
	\begin{equation*}
		K = \{t \in \mathbb{R}^d : a_2 \geq D(t) > a_1 \},
	\end{equation*}
	where $0 < a_1 < a < a_2 < T$ are fixed from now on. 	Therefore, for $0 < \varepsilon < \varepsilon_0$,
	\begin{equation}\label{eq0}
		\Delta_\varepsilon(Y) = \int_K t \frac{\left( \mathbf{1}_{\{ D(t) + \varepsilon Y(t) \geq a \}} - \mathbf{1}_{\{ D(t) \geq a \}} \right)}{\varepsilon} f(t) dt.
	\end{equation}
	
	Denoting by $\nu_x$ the first coordinate of the vector $\nu$:
	\begin{multline*}
		\Delta_\varepsilon(Y) =	\int_K \tau(\tau^{-1}(t)) \left( \mathbf{1}_{\{ (\tau^{-1}(t))_x + \varepsilon Y(\tau(\tau^{-1}(t))) \geq a \}} - \mathbf{1}_{\{ (\tau^{-1}(t))_x \geq a \}} \right) f(\tau(\tau^{-1}(t))) \\
		\times \left| \det(d \tau^{-1}_t) \right| \left| \det(d \tau)_{\tau^{-1}(t)} \right| dt.
	\end{multline*}
	
	Using change of variables $ (\delta, \vec{n}) = \tau^{-1}(t)$, see \citet{lang2012real}:
	\begin{equation}\label{chvar}
		\Delta_\varepsilon(Y) =	\int_{\tau^{-1}(K)} \tau(\delta, \vec{n}) \frac{\mathbf{1}_{\{ \delta + \varepsilon Y(\tau(\delta, \vec{n})) > a \}} - \mathbf{1}_{\{ \delta > a \}}}{\varepsilon} f(\tau(\delta, \vec{n}))	\times \left| \det(d \tau)_{(\delta, \vec{n})} \right| d\delta d\vec{n}.
	\end{equation}
	Then, $\Delta_\varepsilon(Y)$ equals,
	\begin{equation*}
\int_{S^{d-1}} \int_{a_1}^{a_2} \tau(\delta, \vec{n}) \frac{\mathbf{1}_{\{ \delta + \varepsilon Y(\tau(\delta, \vec{n}))\geq  a \}} - \mathbf{1}_{\{ \delta \geq a \}}}{\varepsilon} f(\tau(\delta, \vec{n})) \left| \det(d \tau)_{(\delta, \vec{n})} \right| d\delta d\vec{n}=
		\int_{S^{d-1}} I_\varepsilon(\vec{n}) d\vec{n}.
	\end{equation*}
	Fix $\vec{n} \in S^{d-1}$. Define,
$$G(\delta)  =   \tau(\delta, \vec{n}) f(\tau(\delta, \vec{n})) \left| \det(d \tau)_{(\delta, \vec{n})} \right|,\qquad 
		H(\delta)  = Y(\tau(\delta, \vec{n})),\qquad \text{and} \quad I =  [a_1, a_2].$$
	It is clear that  Lemma  \ref{lem5} applies, and we deduce,
	\begin{equation}
		\lim_{\varepsilon \to 0^+} I_\varepsilon(\vec{n}) = \tau(a, \vec{n}) f(\tau(a, \vec{n})) \left| \det(d \tau)_{(a, \vec{n})} \right| Y(\tau(a, \vec{n})) := I(\vec{n}), \quad \forall \vec{n} \in S^{d-1}.
	\end{equation}
	
	On the other hand
	\[
	\bigl\|I_\varepsilon(\vec{n})\bigr\|
	\;\le\;
	\int_{a_1}^{a_2}
	\bigl\|\tau(\delta,\vec{n})\bigr\|\;
	\mathbf{1}_{\{\{\delta + \varepsilon\,Y(\tau(\delta,\vec{n})) \,\ge\, a\}
		\,\triangle \{\delta,a\}\}}\;
	f\bigl(\tau(\delta,\vec{n})\bigr)\,
	\bigl|\det\bigl(d\tau\bigr)_{(\delta,\vec{n})}\bigr|\,
	d\delta\,ds
	\]
	\[
	\;\le\;
	\int_{a_1}^{a_2}
	\bigl\|\tau(\delta,\vec{n})\bigr\|\;
	\mathbf{1}_{\{\,\delta - \varepsilon\,r < a \,\le\, \delta + \varepsilon\,r\}}
	\;f\bigl(\tau(\delta,\vec{n})\bigr)\,
	\bigl|\det\bigl(d\tau\bigr)_{(\delta,\vec{n})}\bigr|\,
	d\delta\,ds.
	\]
	Where we used that $\|Y\|_\infty \leq r$. If \(0 < \varepsilon < (1/r)\min\{a_2 -a, a - a_1\}$, then
	\[
	\bigl\|I_\varepsilon(\vec{n})\bigr\|\leq \int_{\,a - \varepsilon r}^{\,a + \varepsilon r}
	\bigl\|\tau(\delta,\vec{n})\bigr\|\,
	\bigl|\det\bigl(d\tau\bigr)_{(\delta,\vec{n})}\bigr|\,
	f\bigl(\tau(\delta,\vec{n})\bigr)\,d\delta
	\;=\;
	M_\varepsilon(\vec{n}).
	\]
	Moreover, as \(\varepsilon \to 0^{+}\), $M_\varepsilon(\vec{n})
	\;\longrightarrow\;
	2\,r\,
	\bigl\|\tau(a,\vec{n})\bigr\|\,
	\bigl|\det\bigl(d\tau\bigr)_{\!(a,\vec{n})}\bigr|\,
	f\bigl(\tau(a,\vec{n})\bigr)
	\;=\;
	M(\vec{n}),$
	and by Tonelli's theorem,
	\[
	\int_{S^{d-1}} M_\varepsilon(\vec{n})\,d\vec{n}
	\;=\;\frac{1}{\varepsilon} \int_{\,a - \varepsilon r}^{\,a + \varepsilon r}
	\int_{S^{d-1}}\bigl\|\tau(\delta,\vec{n})\bigr\|\,
	f\bigl(\tau(\delta,\vec{n})\bigr)\,
	\bigl|\det\bigl(d\tau\bigr)_{(\delta,\vec{n})}\bigr|\,
	d\delta\,ds .
	\]
	
	Since for \(\delta \in [a_1,a_2]\), the map $(\delta,\vec{n})
	\;\longmapsto\;
	\|\tau(\delta,\vec{n})\|\,
	f\bigl(\tau(\delta,\vec{n})\bigr)\,
	\bigl|\det(d\tau)\bigr|_{(\delta,\vec{n})},$
	is continuous on the compact domain \([a_1,a_2]\times S^{d-1}\), so its integral is a continuous function of \(\delta\). Then,
	
	\[
	\lim_{\varepsilon \to 0^{+}}
	\,2r
	\int_{S^{d-1}} M_\varepsilon(\vec{n})\,d\vec{n}=\;
	\int_{S^{d-1}} M(\vec{n})\,d\vec{n}.
	\]
	
	From where it follows that,
	\begin{enumerate}
		\item[(i)] \(\|I_\varepsilon(\vec{n})\|\;\le\; M_\varepsilon(\vec{n})\quad
		\forall\,\vec{n}\in S^{d-1},\;\forall\,0<\varepsilon<\varepsilon_{1}.\)
		
		\item[(ii)] \(M_\varepsilon(\vec{n}) \;\xrightarrow[\varepsilon\to 0^{+}]{}\; M(\vec{n})
		\quad \forall\,\vec{n} \in S^{d-1}.\)
		
		\item[(iii)] \(\displaystyle
		\int_{S^{d-1}} M_\varepsilon(\vec{n})\,d\vec{n}
		\;\xrightarrow[\varepsilon\to 0^{+}]{}\;
		\int_{S^{d-1}} M(\vec{n})\,d\vec{n}.
		\;\;\;(\text{B})\)
	\end{enumerate}
 
	From (A) and (B) it follows that
	\[
	\int_{S^{d-1}} I_\varepsilon(\vec{n})\,d\vec{n}
	\;\xrightarrow[\varepsilon\to 0^{+}]{}\;
	\int_{S^{d-1}} I(\vec{n})\,d\vec{n},
	\]
	which concludes the proof that
	\[
	\Delta_\varepsilon(Y)
	\;\xrightarrow[\varepsilon\to 0^{+}]{}\;
	\Delta(Y)
	\quad\forall\,Y\in \mathcal{X},
	\quad(\text{C}).
	\]
Now we will show that this convergence is uniform over \(\mathcal{X}\). Take \(\eta>0\) and choose
	\(Y^1,\dots,Y^{N_\eta} \in \mathcal{X}\)
	such that
	\[
	\sup_{\,Y\in \mathcal{X}}
	\;\min_{\,1 \,\le\, i \,\le\, N_\eta}
	\|\,Y - Y^i\,\|
	\;\le\;\eta.
	\]
	It is clear that
	\[
	\|\Delta(Y) \;-\; \Delta(Y^i)\|
	\;\le\;
	\Bigl(\,
	\int_{S^{d-1}}
	\|\tau(a_i,\vec{n})\|\,
	\bigl|\det(d\tau)\bigr|_{(a_i,\vec{n})}\,
	f\bigl(\tau(a_i,\vec{n})\bigr)
	\,d\vec{n}\Bigr)\;\cdot\;
	\|\,Y \;-\; Y^i\,\|.
	\]
	Since the integral above is finite (denote it by \(C\)), we obtain
	
	\[
	\|\Delta(Y)-\Delta(Y^{i})\|\;\le\;\,C\|Y - Y^{i}\| 
	\quad
	\forall\,Y \in \mathcal{X},\;1 \le i \le N_\eta.  \qquad (D)
	\]
 
	On the other hand,
	\begin{align*}
		\|\Delta_\varepsilon(Y) \;-\; \Delta_\varepsilon(Y^{(i)})\|
		= & \frac{1}{\varepsilon}
		\Bigl\|\,
		\int_{K}
		t\,\bigl(
		\mathbf{1}_{\{\,D(t)+\varepsilon\,Y^{(i)}(t)\,>\,a\}}
		\;-\;
		1_{\{\,D(t)+\varepsilon\,Y(t)\,>\,a\}}
		\bigr)\,f(t)\,dt
		\Bigr\|\\
		\leq & 
		\int_{K}
		\|t\|\,\|f(t)\|\,
		\mathbf{1}_{\{
			a-\varepsilon\,Y^{(i)}(t)\;-\;\varepsilon\,\|Y^{(i)}-Y\|_\infty
			\;\le\;
			D(t)
			\;<\;
			a-\varepsilon\,Y^{(i)}(t)\;+\;\varepsilon\,\|Y^{(i)}-Y\|_\infty
			\bigr\}}
		\frac{dt}{\varepsilon}\\
		\leq &
		\int_{K}
		\|t\|\,\|f(t)\|\,
		\mathbf{1}_{\{
			a-\varepsilon\,Y^{(i)}(t)\;-\;\varepsilon\,\|Y^{(i)}-Y\|_\infty
			< D(t) \le a-\varepsilon\,Y^{(i)}(t) + \varepsilon\,\|Y^{(i)}-Y\|_\infty
			\bigr\}}
		\frac{dt}{\varepsilon}.
	\end{align*} 
	
	Then, if we use the same change of variable as in \eqref{chvar}, $\|\Delta_\varepsilon(Y) \;-\; \Delta_\varepsilon(Y^{(i)})$ equals,
\[\int_{S^{d-1}}
	\int_{a_1}^{a_2}
	\|\tau(\delta,\vec{n})\|\,
	f\bigl(\tau(\delta,\vec{n})\bigr)\,
	\bigl|\det(d\tau)\bigr|_{(a_i,\vec{n})}
	\;\mathbf{1}_{\{\,a - \varepsilon\,Y^{(i)}(\tau(\delta,\vec{n})) - \delta \,\le\,\|Y^{(i)}-Y\|_\infty\}}
	\;d\delta\,d\vec{n}.
	\]
	Observe that 
	$$\{\,a-\varepsilon\,Y^{(i)}-\delta\|=\{Y^i+\|Y^i-Y\|_\infty \geq (a-\delta)/\epsilon \}-\{Y^i-\|Y^i-Y\|_\infty \geq (a-\delta)/\epsilon\}.$$
	Then, 
	\begin{multline*}
		\|\Delta_\varepsilon(Y) \;-\; \Delta_\varepsilon(Y^{(i)})\|=\\
		\int_{S^{d-1}}
		\int_{a_1}^{a_2}
		\|\tau(\delta,\vec{n})\|\,
		f\bigl(\tau(\delta,\vec{n})\bigr)\,
		\bigl|\det(d\tau)\bigr|_{(a_i,\vec{n})} \frac{\mathbf{1}_{\{Y^i(\tau(\delta,\vec{n}))+\|Y^i-Y\|_\infty\geq (a-\delta)/\epsilon\}}-\mathbf{1}_{\{\delta\geq a\}}}{\epsilon}
		d\delta\,d\vec{n}
		\\
		-\int_{S^{d-1}}
		\int_{a_1}^{a_2}
		\|\tau(\delta,\vec{n})\|\,
		f\bigl(\tau(\delta,\vec{n})\bigr)\,
		\bigl|\det(d\tau)\bigr|_{(a_i,\vec{n})} \frac{\mathbf{1}_{\{Y^i(\tau(\delta,\vec{n}))-\|Y^i-Y\|_\infty\geq (a-\delta)/\epsilon\}}-\mathbf{1}_{\{\delta\geq a\}}}{\epsilon} d\delta\,d\vec{n}=:\;B_\varepsilon.
	\end{multline*}
	Applying Lemma \ref{lem5}  to
	\[
	G(\delta)=\|\tau(\delta,\vec{n})\|\,f(\tau(\delta,\vec{n}))\,\bigl|\det(d\tau)\bigr|_{(\delta,\vec{n})},
	\]
	\[
	H^*(\delta) = Y^{(i)}(\tau(\delta,\vec{n})) + \|\,Y^{(i)}-Y\|_\infty, \quad  \text{ and  } \quad \tilde{H}(\delta)= Y^{(i)}(\tau(\delta,\vec{n})) + \|\,Y^{(i)}-Y\|_\infty,
	\]
	and dominating the integrals, we deduce that $B_\varepsilon$ converges, as $\varepsilon\to 0^+$, to 
	\[ \int_{S^{d-1}}
	\|\tau(a_i,\vec{n})\|\,
	f(\tau(a_i,\vec{n}))
	\,\bigl|\det(d\tau)\bigr|_{(a_i,\vec{n})}
	\;\bigl|\,
	\bigl(Y^i(\tau(a_i,\vec{n})) + \|Y^i - Y\|_\infty - Y^i(\tau(a_i,\vec{n})) + \|Y^i - Y\|_\infty\bigr)
	\bigr|\,
	d\vec{n},
	\]
	which implies that 
	\[
	\lim_{\varepsilon\to 0}
	\|\Delta_\varepsilon(Y)\;-\;\Delta_\varepsilon(Y^{\,i})\|
	\;\le\;C\,\|\,Y^i - Y\|_\infty,\quad
	\forall\,Y\in\mathcal{X},\;1\le i\le N_\eta,
	\]
	\noindent
	but, more precisely: (using the monotony of \(B_\varepsilon\) with respect to \(\|Y^i - Y\|_\infty\))
	\[
	\text{If }\|Y^i - Y\|_\infty < h \text{ and } \alpha < \varepsilon < \varepsilon(h)
	\quad \text{then} \quad 
	\|\Delta_\varepsilon(Y)\;-\;\Delta_\varepsilon(Y^{\,i})\|\;\le\;C\,h.
	\qquad(E)
	\]
	Consider now any \(Y\in\mathcal{X}\); take \(Y^i\)  such that  \(\|Y^i - Y\|_\infty < \gamma\), then,  
	\[
	\|\Delta(Y)\;-\;\Delta_\varepsilon(Y)\|
	\;\le\;
	\|\Delta(Y)\;-\;\Delta(Y^{\,i})\|
	\;+\;
	\|\Delta(Y^{\,i})\;-\;\Delta_\varepsilon(Y^{\,i})\|
	\;+\;
	\|\Delta_\varepsilon(Y^{\,i})\;-\;\Delta_\varepsilon(Y)\|.
	\]
	Using (D) and (E) we get that, if $0<\varepsilon<\varepsilon(\gamma)$ then 
	$$ \|\Delta(Y)\;-\;\Delta_\varepsilon(Y)\|
	\;\le\;
	C\,\gamma
	\;+\;
	\|\Delta(Y^{\,i})\;-\;\Delta_\varepsilon(Y^{\,i})\|.$$
	Furthermore, 
	\[
	\sup_{Y\in\mathcal{X}}
	\|\Delta(Y)\;-\;\Delta_\varepsilon(Y)\|
	\;\le\;
	C\,\gamma
	\;+\;
	\max_{1\le i\le N_\gamma}
	\|\Delta(Y^{\,i})\;-\;\Delta_\varepsilon(Y^{\,i})\|.
	\]
	Since $ \Delta_\varepsilon(Y^{\,i})\,\to\,  \Delta(Y^{\,i})$ as $\varepsilon\to 0^+$ $\forall\,1\le i\le N_\eta$  and $N_\eta$ is finite,	
	$$ \overline{\lim_{\varepsilon\to 0}}
	\sup_{Y\in\mathcal{X}}
	\|\Delta(Y)\;-\;\Delta_\varepsilon(Y)\|
	\;\le\;
	C\,\gamma.$$
	Taking now $\gamma\to 0^+$,   the result follows.
\end{proof}

\begin{prop}\label{prop2}
	Let $\mathcal{G}$ be a compact subset of 
	$$
	B_{PD}(r)\;=\Big\{\,g:\mathbb{R}^d \to \mathbb{R}\;:\;g\in L^1(\mathbb{R}^d)  \text{ is bounded and continuous on }\mathcal{U},  \|g\|_\infty \le r\Big\}.$$
	Assume \textbf{(H1)} to \textbf{(H4)} and let $Y\in\mathcal{X}$, $\mathcal{X}$ be as in Proposition \ref{prop1}.
	Let
	\[
	\nabla_{\varepsilon}(g)
	\;=\;
	\int_{\mathbb{R}^d}
	t\;
	\mathbf{1}_{\{\;D(t)+\varepsilon\,Y(t) \geq  a\}}\;
	g(t)\;dt,
	\qquad \text{ and }\qquad 
	\nabla(g)
	\;=\;
	\int_{\mathbb{R}^d}
	t\;
	\mathbf{1}_{\{\;D(t) \geq  a\}}\;
	g(t)\;dt.
	\]
	Then,
	\[
	\sup_{\,Y\in\mathcal{X},\,g\in\mathcal{G}}
	\|\nabla_{\varepsilon}(g) \;-\; \nabla(g)\|
	\;\xrightarrow[\varepsilon\to 0^{+}]{}\;0.
	\]
\end{prop}
\begin{proof}
	Using again the remark that the integral can be reduced to a compact set, for a fixed $g$, $\nabla_{\varepsilon}(g)\to \nabla(g)$ as $\varepsilon\to 0^{+}$, as a trivial consequence of Dominated Convergence Theorem.	
	Consider now $\eta > 0$, and $g_1,\dots,g_{N_\eta}\in \mathcal{G}$ such that
	\[
	\sup_{\,g\in\mathcal{G}}\;
	\min_{\,1 \le i \le N_\eta}\;\|\,g - g_i\|_\infty
	\;<\;\eta.
	\]
Clearly,
	\[
	\|\nabla(g) - \nabla(g_i)\|
	\;\le\;
	C\,\|\,g - g_i\|_\infty.
	\]
	Let $K$ be the compact set 
	\(\{\,t : D(t)\in [a_1,a_2]\, 0<a_1<a<a_2<T\}\),
	and so on. Then
\begin{multline*}
	\|\nabla_{\varepsilon}(g) \;-\; \nabla_{\varepsilon}(g_i)\|\leq \int_{K} \|t\|\mathbf{1}_{\{\;D(t)+\varepsilon\,Y(t) \geq  a\}}dt\|g-g_i\|\leq \int_{K} \|t\|\mathbf{1}_{\{\;D(t)+\varepsilon r \geq  a\}}dt\|g-g_i\|\\
	 \le\;C(r)\,\|\,g - g^i\|.
\end{multline*}
	
	It follows that, for $\varepsilon$ small,
\begin{multline*}
		\|\nabla_\varepsilon(g)\;-\;\nabla(g)\|
	\;\le\;
	C(r)\,\|\,g-g_i\|
	\;+\;
	\|\nabla_\varepsilon(g_i)\;-\;\nabla(g_i)\|
	\;+\;
	C(r)\,\|\,g-g_i\|\\
	=C(r)\eta+\;\|\nabla_\varepsilon(g^i)\;-\;\nabla(g^i)\|. 
\end{multline*}
	Then, 
	\[\sup_{\,g\in\mathcal{G}}
	\|\nabla_\varepsilon(g)\;-\;\nabla(g)\|
	\;\le\;
	C(r)\,\eta
	\;+\;
	\max_{\,1\le i\le N_2}
	\|\nabla_\varepsilon(g_i)\;-\;\nabla(g_i)\|.
	\]
	From where it follows,
	\[
	\lim_{\varepsilon\to 0}
	\sup_{\,g\in\mathcal{G}}
	\|\nabla_\varepsilon(g)\;-\;\nabla(g)\|
	\;\le\;
	Q(r)\,\eta.	\]
	Taking $\eta\to 0^+$, the result follows. 
\end{proof}
\begin{thm}\label{thhaddamard}
	Let $\mathcal{U}$ as in \textbf{(H4)} and 
	\begin{align*}
		\mathcal{D}
		\;=\; & 	\{\,y:\mathbb{R}^d \to \mathbb{R}\;\text{continuous on } \mathcal{U}, \text{ and bounded}\},\\
		\mathcal{D}\mathcal{G}
		\;=\; & 	\{\,g:\mathbb{R}^d \to \mathbb{R}\;\text{continuous on } \mathcal{U}, \text{ and  bounded},\;g\in L^1(\mathbb{R}^d)\}.
	\end{align*}

	Let $D\in\mathcal{D}$ and $f\in\mathcal{D}\mathcal{G}$ such that \textbf{(H1)} to \textbf{(H4)} hold. 
	Let $0<a<T$ and $a_1 < a < a_2 < T$  such that   $K=\{\,t\in\mathbb{R}^d : a_1 \le D(t) < a_2\}\subset \cup_{i=1}^\ell \Omega_i$, \eqref{menormu} and \eqref{disjoint} holds. The functional
	\[
	\mathcal{T}(\tilde{y},g)
	\;=\;
	\int_{K}
	t\,
	1_{\{\,\tilde{y}(t)>a\}}
	\;g(t)\;dt
	\quad
	\text{for any }(\tilde{y},g)\in \mathcal{D}\times \mathcal{D}\mathcal{G},
	\]
	is Hadamard differentiable in $(D,f)$, and

	\[
	\mathcal{T}'(\tilde{y},g)
	\;=\;
	\sum_{i=1}^\ell\Bigg\{
	\int_{S^{d-1}}
	\tau(a,\vec{n})\,
	f\bigl(\tau_i(a,\vec{n})\bigr)\,
	\left|\det\!\bigl(d_{(a,\vec{n})}\tau_i\bigr)\right|\,
	\tilde{y}\bigl(\tau_i(a,\vec{n})\bigr)
	\,d\vec{n}
	\;+\;
	\int_{\mathbb{R}^d}
	t\,
	\mathbf{1}_{\{D(t)\geq a\}}
	\,g(t)\,dt\Bigg\}.
	\]
	
\end{thm}
\begin{proof} As before, we will assume $\ell=1$, the general case is proven exactly the same.
\begin{multline*}
		\frac{\mathcal{T}\bigl(D+\varepsilon \tilde{y},\;f+\varepsilon g\bigr) \;-\; \mathcal{T}(D,f)}{\varepsilon}
	\;=\;
	\int_{K}
	t\,\frac{\mathbf{1}_{\{D(t) + \varepsilon \tilde{y}(t) \geq  a\}}
		\;-\;
		\mathbf{1}_{\{D(t)\geq  a\}}}{\varepsilon}\,
	f(t)\,dt
	\;+\\
	\int_{K}
	t\,
	\mathbf{1}_{\{D(t) + \varepsilon \tilde{y}(t)\geq a\}}
	\,g(t)\,dt. 
\end{multline*}
 
	The first term, as \(\varepsilon \to 0\), it converges uniformly on compact sets to
	\[
	\int_{S^{d-1}}
	\tau(a,\vec{n})\,
	f\bigl(\tau(a,\vec{n})\bigr)\,
	\left|\det\!\bigl(d\tau_{(a,\vec{n})}\bigr)\right|\,
	\tilde{y}\bigl(\tau(a,\vec{n})\bigr)
	\,d\vec{n},
	\]
	by Proposition \ref{prop1}.
	The second term, it converges to
	\[
	\int_{\mathbb{R}^d}
	t\,
	\mathbf{1}_{\{D(t) \geq  a\}}
	\,g(t)\,dt.
	\]
	by Proposition \ref{prop2}. 	Hence the limit of the difference is exactly the derivative claimed.
\end{proof}
\medskip
\noindent
\begin{remark} \label{rem9} If \(\|Y\|_{\infty}\le r\), then for \(\varepsilon\) small enough,
	\begin{multline*}
		\frac{\mathcal{T}\bigl(D + \varepsilon Y,\;f + \varepsilon g\bigr)
			\;-\;
			\mathcal{T}(D,f)
			\;}{ \varepsilon }=\;
		\int_{\mathbb{R}^d}
		t\,
		\frac{\mathbf{1}_{\{D(t) + \varepsilon Y(t) \geq  a\}}
			\;-\;
			\mathbf{1}_{\{D(t) \geq  a\}}}{\varepsilon}
		\,f(t)\,dt
		\;+\; \\
		\int_{\mathbb{R}^d}
		t\,
		\mathbf{1}_{\{D(t)+\varepsilon Y(t) \geq a\}}
		\,g(t)\,dt.
	\end{multline*}
	
\end{remark}

\section{Asymptotic distribution for trimmed means}\label{tclsec}

In this section, we derive the asymptotic distribution of the trimmed mean estimator \eqref{estimador} for the case $D_n(x)=D(x,\hat{f}_n)$. As before, the theorem is stated for the case $\ell=1$; however, the proof follows the same ideas in the general case. Hypothesis \textbf{(H4)} is assumed to hold for $p=2$.

The following theorem assumes that $\hat{f}_n$ is a sequence of continuous functions that serve as estimators of $f$ (not necessarily kernel-based).

\begin{thm}\label{tcl} Let $D$ be such that \textbf{(H1)} to \textbf{(H4})  holds.  
	Assume also that 
	\begin{itemize}
		\item[(i)] $D$ is Hadamard differentiable with respect to $(x,f)$, and that
		\[
		\frac{\partial D}{\partial g}(x,g),
		\]
		is continuous as a function of $(x,g)$ in a neighborhood of $f$.   
		\item[(ii)] $\|\hat{f}_n-f\|_{\infty}\to 0$ a.s on compact sets, and 
		\item[(iii)] $\sqrt{a_n}\bigl(\hat{f}_n - f\bigr)(t)$ converges weakly to $b(t)$, on $\{t:D(t)\geq a\}$, $b(t)$ being a stochastic process indexed on $\{t:D(t)\geq a\}$.
	\end{itemize}
	Let us denote
	$$\Pi(D_n,\hat{f}_n)=\int_{\mathbb{R}^d} t\mathbf{1}_{\{D_n(t)\geq a\}}\hat{f}_n(t)dt\qquad \text{and} \qquad \Pi(D,f)=\int_{\mathbb{R}^d} t\mathbf{1}_{\{D(t)\geq a\}}f(t)dt,$$
	then $\sqrt{a_n}\,\bigl(\Pi(D_n,f_n)\;-\Pi(D,f)\bigr)$ converges weakly to 
	\begin{equation} \label{eqthm}
		\int_{S^{d-1}}
		\tau(a,\vec{n}) \;
		f\bigl(\tau(a,\vec{n})\bigr)\;
		\frac{\partial D}{\partial g}\bigl(\tau(a,\vec{n}),f\bigr)b(\tau(a,\vec{n}))	\left|\det\!\bigl(d\tau_{(a,\vec{n})}\bigr)\right|
		\,d\vec{n}
		\;+\;
		\int_{\mathbb{R}^d}
		t\;\mathbf{1}_{\{D(t)\geq a\}}\;b(t)\,dt.
	\end{equation}
\end{thm}

\begin{remark}
	Assumptions \textbf{(H1)} through \textbf{(H4)} were discussed in Remark \ref{remhipo}.  Assumption (i) is satisfied, for example, by Tukey’s depth (see Lemma 5.6 in \citet{masse2009}). Assumption (ii) is discussed in detail in point 1 of Remark \ref{remhipo}. Regarding assumption (iii) see Corollary   in \citet{Rosenblatt1976}. 	
\end{remark}
\begin{proof} 
	Let us observe that,
	\begin{align*}
		\Pi(D_n,\hat{f}_n)\;-\;\Pi(D,f)
		\;=\; & 	 \int_{\mathbb{R}^d}
		t\mathbf{1}_{\{D_n(t)\geq a\}} \hat{f}_n(t)\,dt
		-
		\int_{\mathbb{R}^d}
		t\,\mathbf{1}_{\{D(t)\geq a\}}\, f(t)dt\\
		=&\int_{\mathbb{R}^d} t(\mathbf{1}_{\{D_n(t)\geq a \}}-\mathbf{1}_{\{D(t)\geq a \}})f(t)dt+ \int_{\mathbb{R}^d}\mathbf{1}_{\{D_n(t)\geq a \}}(\hat{f}_n(t)-f(t))dt.
	\end{align*}
	By Remark    \ref{rem9}, for $n$ large enough, $\Pi(D_n,\hat{f}_n)\;-\;\Pi(D,f)=\mathcal{T}(D_n,f_n)\;-\;\mathcal{T}(D,f).$
	By the Hadamard differentiability of $T$ at $(D,f)$,
	\[	\mathcal{T}(D_n,\hat{f}_n)\;-\;\mathcal{T}(D,f)
	\;=\;
	\mathcal{T}'(D,f)\,\bigl(D_n-D,\;\hat{f}_n-f\bigr)
	\;+\;
	o\bigl(\|D_n-D\|_{\infty}+\|\hat{f}_n-f\|_{\infty}\bigr).
	\]

	Since $D$ is Hadamard differentiable w.r.t. the second component, and the fact that
	\[
	\frac{\partial D}{\partial g}(x,g)
	\]
	is a continuous function of $(x,g)$ in a neighborhood of $f$,  $\|D_n-D\|_\infty\approx C\|\hat{f}_n-f\|_\infty$ where the norm is on compact sets. Then, 	
	\[
	\sqrt{a_n}\,\bigl(\mathcal{T}(D_n,\hat{f}_n)\;-\;\mathcal{T}(D,f)\bigr)
	\;\approx\;
	\sqrt{a_n}\,\bigl[\,
	\mathcal{T}'(D,f)\,\bigl(D_n-D,\;\hat{f}_n-f\bigr)
	\bigr].
	\]
	By Theorem \eqref{thhaddamard}
	\begin{multline*}
		\sqrt{a_n}\,\bigl(\mathcal{T}(D_n,\hat{f}_n)\;-\;\mathcal{T}(D,f)\bigr)
		\;\approx\;
		\sqrt{a_n}
		\int_{S^{d-1}}
		\tau(a,\vec{n})
		f\bigl(\tau(a,\vec{n})\bigr)\,
		\bigl|\det\!\bigl(d\tau_{(a,\vec{n})}\bigr)\bigr|\,
		\bigl(D_n - D\bigr)\bigl(\tau(a,\vec{n})\bigr)\,d\vec{n}
		\;+\;\\
		\sqrt{a_n}
		\int_{\mathbb{R}^d}
		t\,\mathbf{1}_{\{D(t)>a\}}\,
		\bigl(\hat{f}_n(t)-f(t)\bigr)\,dt.
	\end{multline*}
	
	Since 
	$$(D_n-D)(x)=D(x,\hat{f}_n)-D(x,f)=\frac{\partial D}{\partial g}(x,f)(\hat{f}_n-f)+o(\|\hat{f}_n-f\|_\infty),$$
	it follows that,
	\begin{multline*}\label{final}
		\sqrt{a_n}\,\bigl(\mathcal{T}(D_n,\hat{f}_n)\;-\;\mathcal{T}(D,f)\bigr)
		\;\approx\;\\
		\sqrt{a_n}
		\int_{S^{d-1}}
		\tau(a,\vec{n})
		f\bigl(\tau(a,\vec{n})\bigr)\,
		\bigl|\det\!\bigl(d\tau_{(a,\vec{n})}\bigr)\bigr|\,
		\frac{\partial D}{\partial g}(\tau(a,\vec{n}),f)(\hat{f}_n-f)\bigl(\tau(a,\vec{n})\bigr)\,d\vec{n}
		\;+\;\\
		\sqrt{a_n}
		\int_{\mathbb{R}^d}
		t\,\mathbf{1}_{\{D(t)>a\}}\,
		\bigl(\hat{f}_n(t)-f(t)\bigr)\,dt.
	\end{multline*}
	 
	We define two linear operators:
	$$L_1(g)= \int_{S^{d-1}} w_1(\vec{n})\, g\bigl(\tau(a,\vec{n})\bigr) \, d\vec{n},\qquad \text{ and }\qquad 
	L_2(g)= \int_{\mathbb{R}^d} w_2(t)\, g(t)\, dt,$$
	with \(w_1\) and \(w_2\) being the (continuous) weight functions. Since  $\sqrt{a_n}\bigl(\hat{f}_n - f\bigr) $ converges weakly  uniformly,  to a process  $b$ defined on the compact set $\{t:D(t)\geq a\}$.
	
	\[
	L_1\Bigl(a_n\bigl(\hat{f}_n-f\bigr)\Bigr)\overset{\mathcal{L}}{\longrightarrow} L_1(b)
	\quad\text{and}\quad
	L_2\Bigl(a_n\bigl(\hat{f}_n-f\bigr)\Bigr) \overset{\mathcal{L}}{\longrightarrow} L_2(b).
	\]
	where $\overset{\mathcal{L}}{\longrightarrow}$ denotes weak convergence. Since \(L_1\) and \(L_2\) are continuous linear functionals defined on the same space, the vector
	\[
	\Bigl( L_1\bigl(a_n\bigl(\hat{f}_n-f\bigr)\bigr),\; L_2\bigl(a_n\bigl(\hat{f}_n-f\bigr)\bigr) \Bigr)
	,\]
	converges in distribution to $( L_1(b),\; L_2(b))$. Then, By applying the continuous mapping theorem 
	\[
	L_1\bigl(a_n\bigl(\hat{f}_n-f\bigr)\bigr) + L_2\bigl(a_n\bigl(\hat{f}_n-f\bigr)\bigr)
	\overset{\mathcal{L}}{\longrightarrow}L_1(b) + L_2(b).
	\]

\end{proof}

\subsection{Simulations}\label{simus}

This subsection presents a simulation example that illustrates the limiting distribution described in Theorem~\ref{tcl}. To visualize the shape of the limiting distribution due to its complexity, $R_n = \sqrt{a_n}\left( \Pi(D_n, f_n) - \Pi(D, f) \right)$ values are simulated.  For each measure, $500$ values are simulated from the distribution of $R_n$ using two values of \( a \) (selected based on the range of values for each depth measure) and assuming \( a_n = \sqrt{n} \).

Each simulation of \( R_n \) is based on a sample of size \( n = 500 \) drawn from a bivariate Beta distribution with independent components, where each component follows a \(\text{Beta}(2,2)\) distribution. The density estimate \(\hat{f}_n\) is computed using kernel density estimation, with the bandwidth \( h \) chosen according to Silverman’s rule of thumb, see \citet{silverman1986}. In all cases we choose the gaussian kernel. Three classic depth measures are considered: Liu's depth, Tukey's depth, and a projection-based depth measure. Figures \ref{fig:liu}, \ref{fig:tukey}, and \ref{fig:proj} display the estimated densities and their corresponding contour curves for each depth measure and each level of \( a \).

\begin{figure}[htbp]
    \centering

    \begin{subfigure}{\textwidth}
        \centering
        \begin{minipage}{0.48\textwidth}
            \centering
            \includegraphics[width=\textwidth]{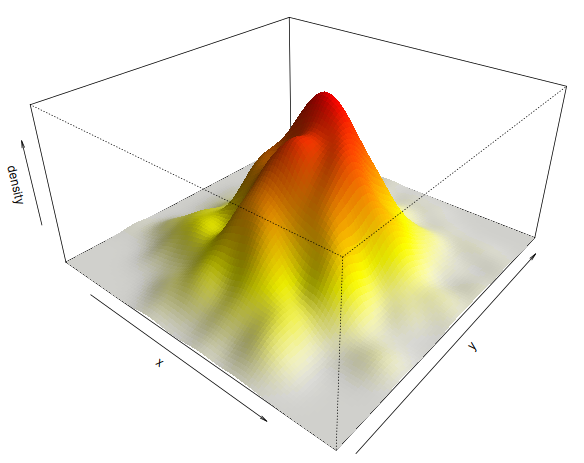}
        \end{minipage}
        \hfill
        \begin{minipage}{0.48\textwidth}
            \centering
            \includegraphics[width=\textwidth]{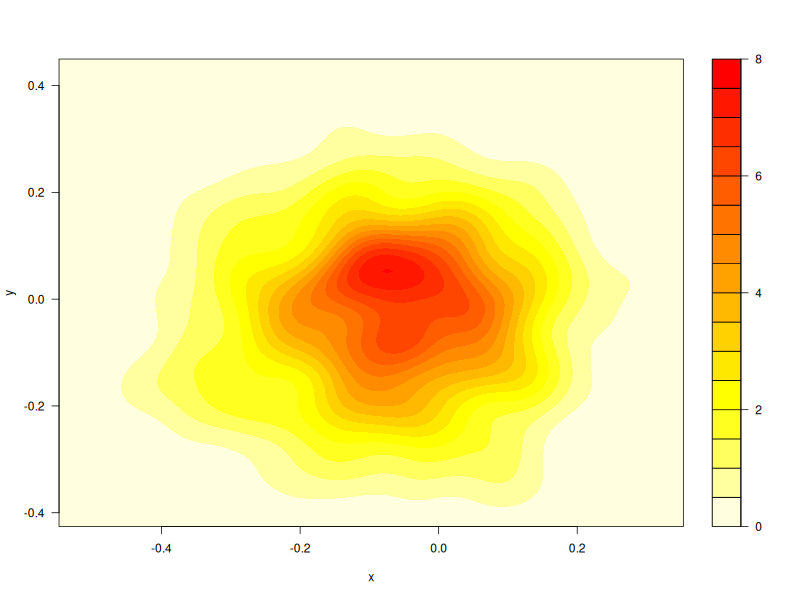}
        \end{minipage}
        \caption{$a=0.05$}
    \end{subfigure}

    \vspace{0.5cm} 

    \begin{subfigure}{\textwidth}
        \centering
        \begin{minipage}{0.48\textwidth}
            \centering
            \includegraphics[width=\textwidth]{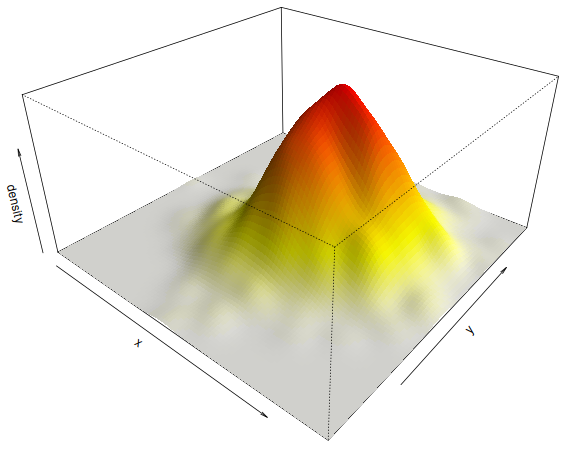}
        \end{minipage}
        \hfill
        \begin{minipage}{0.48\textwidth}
            \centering
            \includegraphics[width=\textwidth]{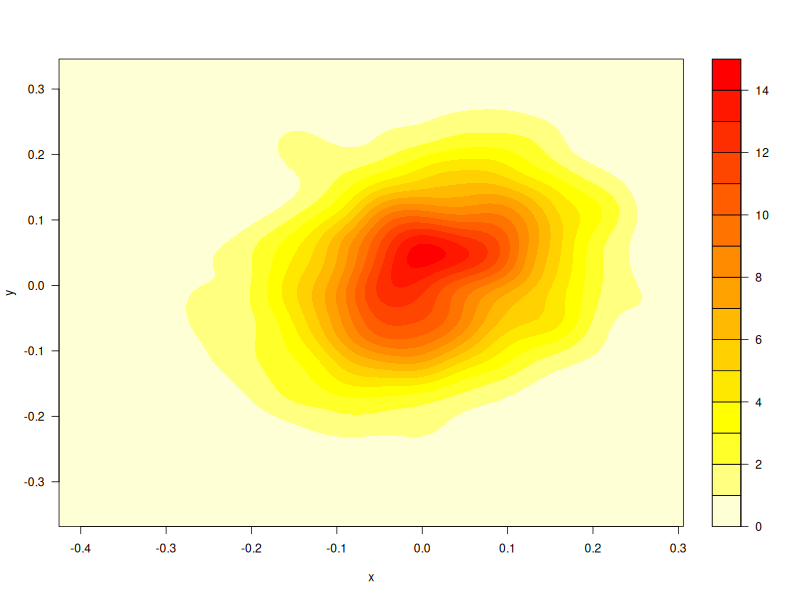}
        \end{minipage}
        \caption{$a=0.1$}
    \end{subfigure}

    \caption{Density estimation (left panel) and corresponding level sets (right panel) of the distribution of $\sqrt{a_n} \left( \Pi(D_n, f_n) - \Pi(D, f) \right)$ for a sample of size $500$, computed using \textit{Liu depth} for different values of $a$.
}
    \label{fig:liu}
\end{figure}

\begin{figure}[htbp]
    \centering

    \begin{subfigure}{\textwidth}
        \centering
        \begin{minipage}{0.48\textwidth}
            \centering
            \includegraphics[width=\textwidth]{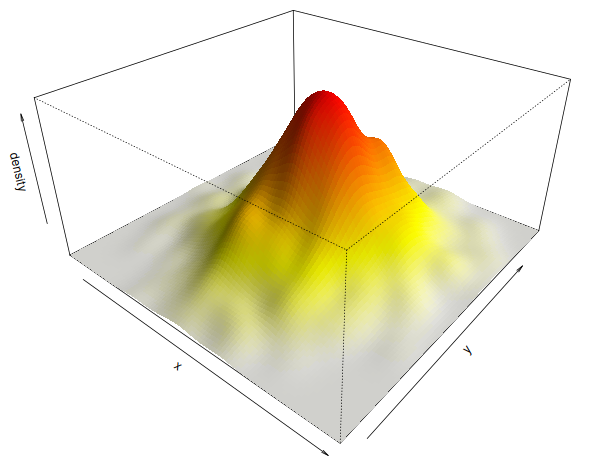}
        \end{minipage}
        \hfill
        \begin{minipage}{0.48\textwidth}
            \centering
            \includegraphics[width=\textwidth]{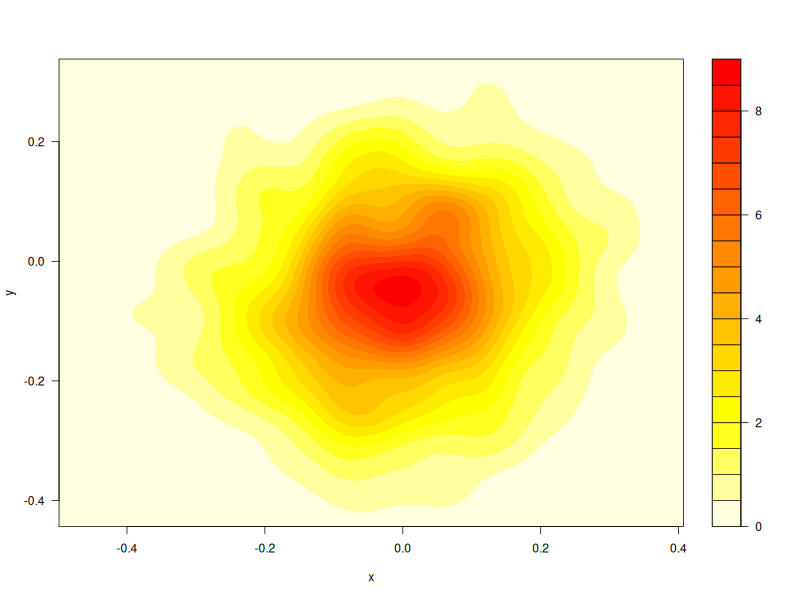}
        \end{minipage}
        \caption{$a=0.1$}
    \end{subfigure}

    \vspace{0.5cm} 

    \begin{subfigure}{\textwidth}
        \centering
        \begin{minipage}{0.48\textwidth}
            \centering
            \includegraphics[width=\textwidth]{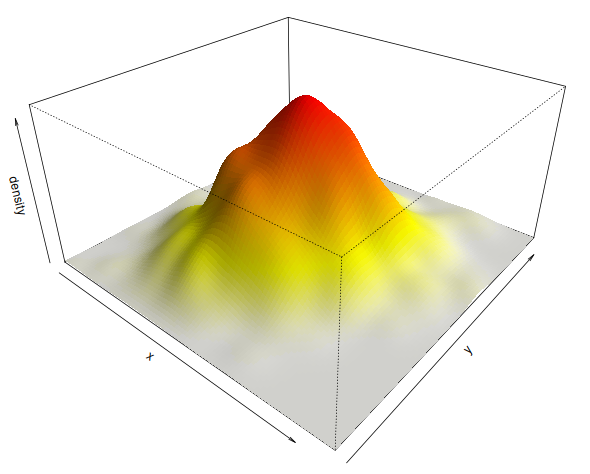}
        \end{minipage}
        \hfill
        \begin{minipage}{0.48\textwidth}
            \centering
            \includegraphics[width=\textwidth]{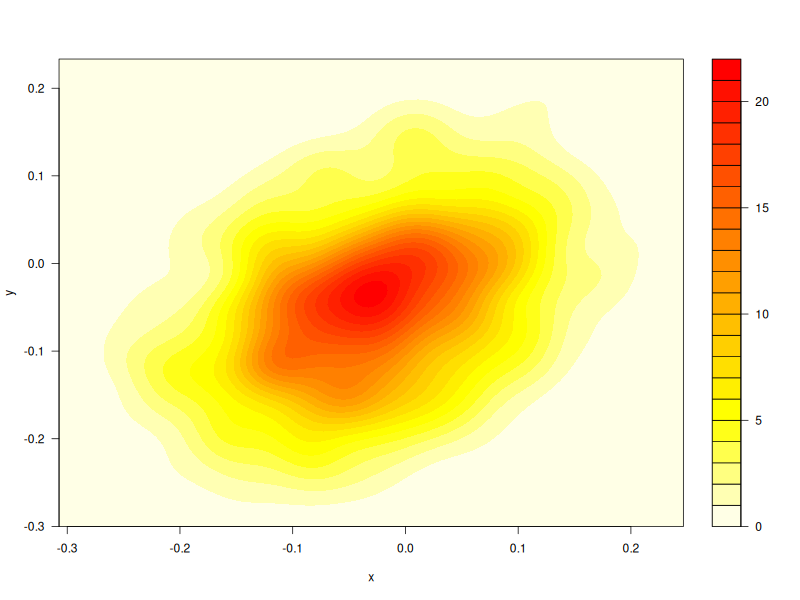}
        \end{minipage}
        
        \caption{$a=0.2$
}
    \end{subfigure}

    \caption{Density estimation (left panel) and corresponding level sets (right panel) of the distribution of $\sqrt{a_n} \left( \Pi(D_n, f_n) - \Pi(D, f) \right)$ for a sample of size $500$, computed using \textit{Tukey depth} for different values of $a$.}
    \label{fig:tukey}
\end{figure}

\begin{figure}[htbp]
    \centering

    \begin{subfigure}{\textwidth}
        \centering
        \begin{minipage}{0.48\textwidth}
            \centering
            \includegraphics[width=\textwidth]{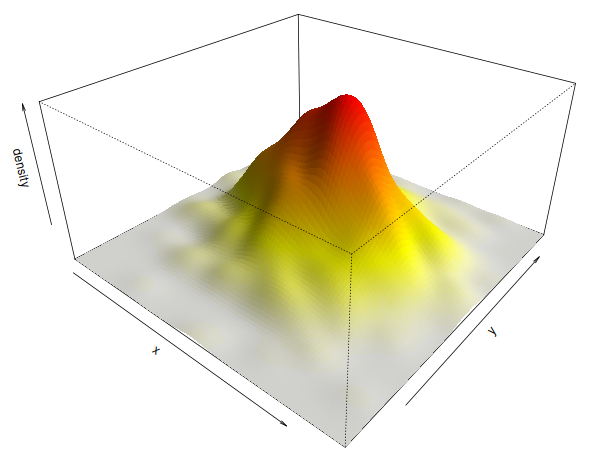}
        \end{minipage}
        \hfill
        \begin{minipage}{0.48\textwidth}
            \centering
            \includegraphics[width=\textwidth]{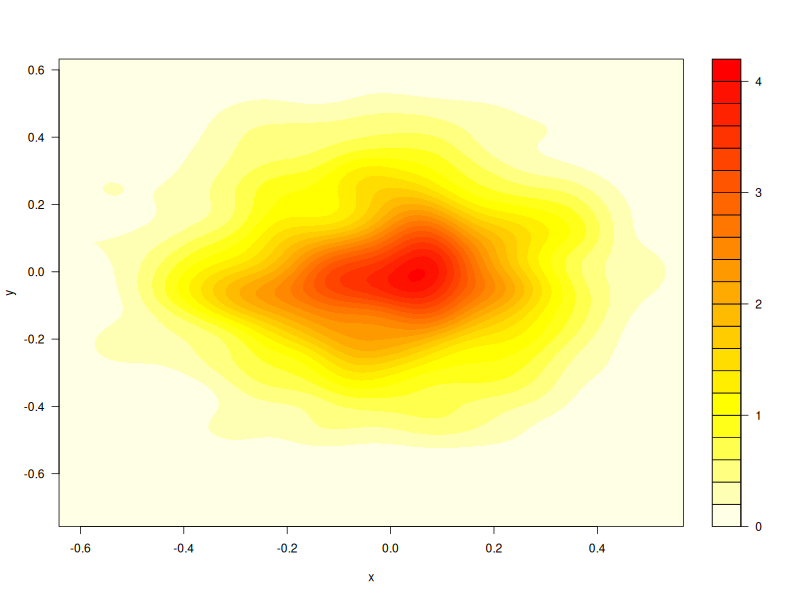}
        \end{minipage}
        \caption{$a=0.1$}
    \end{subfigure}

    \vspace{0.5cm} 

    \begin{subfigure}{\textwidth}
        \centering
        \begin{minipage}{0.48\textwidth}
            \centering
            \includegraphics[width=\textwidth]{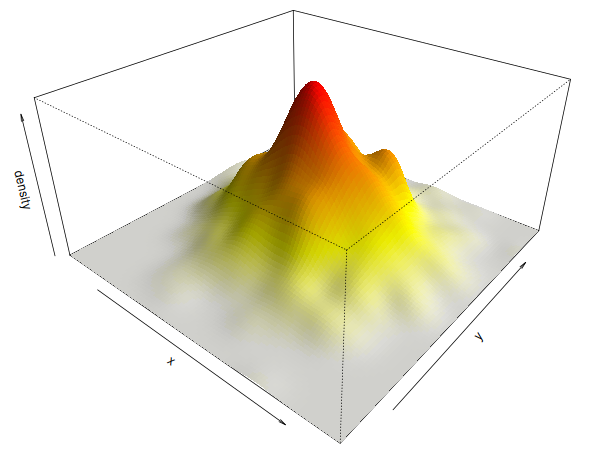}
        \end{minipage}
        \hfill
        \begin{minipage}{0.48\textwidth}
            \centering
            \includegraphics[width=\textwidth]{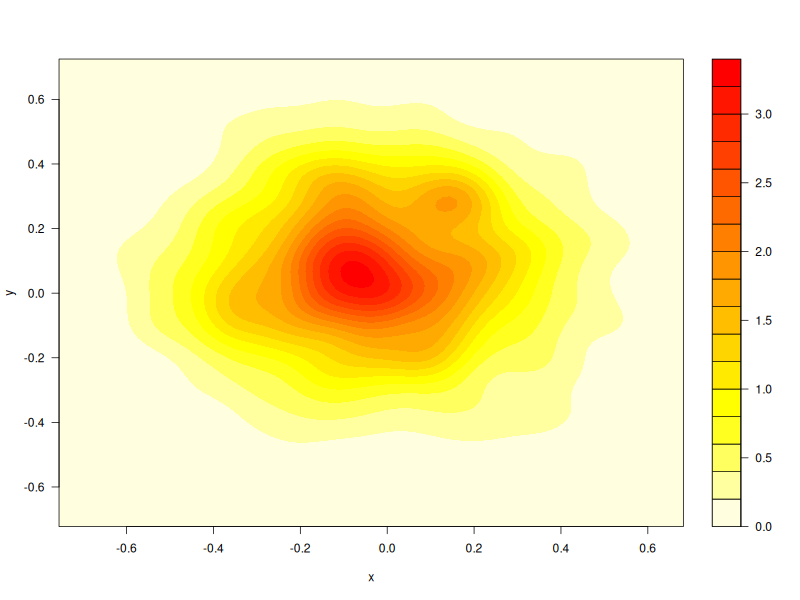}
        \end{minipage}
        \caption{$a=0.2$}
    \end{subfigure}

    \caption{Density estimation (left panel) and corresponding level sets (right panel) of the distribution of $\sqrt{a_n} \left( \Pi(D_n, f_n) - \Pi(D, f) \right)$ for a sample of size $500$, computed using \textit{Projection depth} for different values of $a$.
}
    \label{fig:proj}
\end{figure}

\end{document}